\documentclass{article}[12pt]

\usepackage{amsmath, amsthm, amssymb, amsfonts, mathtools, bbm}                                                                                                                                               
\usepackage{enumerate}
\usepackage{mathrsfs}

\newtheorem{theorem}{Theorem}
\newtheorem{lemma}{Lemma}

\theoremstyle{remark}

\newtheorem{remark}{Remark}
              
\title{A Sharp Entropy Condition For The Density Of Angular Derivatives}
\author{Alex Bergman}
\date{\today}

\begin{document}

\maketitle

\begin{abstract}
    Let $f$ be a holomorphic self-map of the unit disc. We show that if $\log (1-\lvert f(z) \rvert)$ is integrable on a sub-arc of the unit circle, $I$, then the set of points where the function f has finite Carathéodory angular derivative on I is a countable union of Beurling-Carleson sets of finite entropy. Conversely, given a countable union of Beurling-Carleson sets, $E$, we construct a holomorphic self-map of the unit disc, $f$, such that the set of points where the function has finite Carathéodory angular derivative is equal to $E$ and $\log(1-\lvert f(z) \rvert)$ is integrable on the unit circle. Our main technical tools are the Aleksandrov disintegration Theorem and a characterization of countable unions of Beurling-Carleson sets due to Makarov and Nikolski.
\end{abstract}

\section{Introduction}

In this note $\mathbb{D}$ is the open unit disc in the complex plane, $\mathbb{T}$ is the unit circle, and $m$ is the normalized Lebesgue measure on $\mathbb{T}$. A holomorphic self-map of the unit disc, $f : \mathbb{D} \to \mathbb{D}$, is said to have angular derivative at $\lambda \in \mathbb{T}$ in the sense of Carathéodory if the following two conditions hold
\begin{enumerate}[(i)]
    \item $f$ has non-tangential limit at $\lambda$ and, $f(\lambda) = \lim_{r \to 1^{-}} f(r\lambda)$, satisfies $\lvert f(\lambda)\rvert = 1$.
    \item The derivative, $f'$, has finite non-tangential limit at $\lambda$.
\end{enumerate}
The set of points where the function $f$ has finite Carathéodory angular derivative will be denoted $C(f)$. Carathéodory angular derivatives play an important role in several areas of analysis. In particular, they are important in perturbation theory of self-adjoint and unitary operators \cite{simon1993spectral}, conformal mapping \cite{MR2150803}, and model and de Branges-Rovnyak spaces \cite{MR1289670}. The above list of topics and references is far from exhaustive. We shall be concerned with holomorphic self-maps of the unit disc which are \textit{locally non-extreme} in the following sense
\begin{equation}\label{eq:cond}
    \int_{I} \log (1-\lvert f (z)\rvert)dm(z) > -\infty,
\end{equation}
where $I \subset \mathbb{T}$ is a relatively open arc. If \eqref{eq:cond} holds then $\lvert f \rvert < 1$ almost-everywhere on $I$ and hence $C(f) \cap I$ has measure $0$. Our main result is that $C(f) \cap I$ is in fact much smaller than a generic set of measure $0$. When $I = \mathbb{T}$ condition \eqref{eq:cond} is equivalent to $f$ being a non-extreme point of the unit ball of $H^{\infty}$, here $H^{\infty}$ is the set of bounded holomorphic functions on the unit disc. This condition plays a crucial role in the theory of de Branges-Rovnyak spaces, see \cite{MR1289670}.

We now introduce the relevant entropy condition. A closed set $E \subset \mathbb{T}$ of measure zero, $\lvert E \rvert = 0$, is called a Beurling-Carleson set of finite entropy if
\begin{equation}\label{eq:BC_cond}
    \sum_{I} \lvert I \rvert \log \left( 1 / \lvert I \rvert \right) < \infty,
\end{equation}
where the sum is taken over the maximal complementary arcs of $E$, here and in the sequel $\lvert E \rvert = m(E)$. Beurling-Carleson sets have appeared in connection with many problems in complex analysis. To be brief we only mention that they are the zero sets of analytic functions in the disc which extend to smooth functions on the circle, \cite{Taylor_Williams_1970}.

Our main result is the following characterization of $C(f)$ for locally non-extreme holomorphic self-maps of the unit disc.

\begin{theorem}\label{thm:local_ang_deriv}
    Let $f : \mathbb{D} \to \mathbb{D}$ be holomorphic and $I \subset \mathbb{T}$ an open arc. If
    \begin{equation*}
        \int_{I} \log(1-\lvert f\rvert) dm > -\infty.
    \end{equation*}
    Then $C(f) \cap I$ is a countable union of Beurling-Carleson sets. Conversely, if $E \subset \mathbb{T}$ is a countable union of Beurling-Carleson sets there exists an analytic function $f : \mathbb{D} \to \mathbb{D}$ satisfying
    \begin{equation*}
        \int_{\mathbb{T}} \log(1-\lvert f\rvert) dm > -\infty,
    \end{equation*}
    and $C(f) = E$.
\end{theorem}

The sum in \eqref{eq:BC_cond} increases the more spread out the points of $E$ are. Thus the Beurling-Carleson condition is telling us that the Carathéodory angular derivatives of $f$ must have a certain amount of clumping.

Our proof of Theorem \ref{thm:local_ang_deriv} is based on the Aleksandrov disintegration Theorem and a characterization of countable unions of Beurling-Carleson sets by Makarov and Nikolski. We discuss these results in the next section.

\subsection*{Acknowledgements}
    The author wishes to thank Adem Limani and the anonomous referee for several helpful comments.

\section{Preliminaries}\label{sec:prelims}

    For each $\alpha \in \mathbb{T}$ let $\mu_{\alpha}$ be the measure given by the Herglotz integral formula
\begin{equation*}
    \frac{1 - \lvert f(z) \rvert^{2}}{\lvert \alpha - f(z) \rvert^{2}} = \int \frac{1-\lvert z \rvert^{2}}{\lvert \zeta -z\rvert^{2}}d\mu_{\alpha}(\zeta).
\end{equation*}
The measure $\mu_{\alpha}$ is of finite mass and positive. The measures $\left\{ \mu_{\alpha} \right\}_{\alpha \in \mathbb{T}}$ are called the Aleksandrov-Clark measures of $f$. Information about Aleksandrov-Clark measures can be found in, for example, the survey article \cite{MR2198367} and the notes \cite{MR2394657}. Let $(\mu_{\alpha})_{a}$ denote the absolutely continuous part of the measure $\mu_{\alpha}$. Standard properties of Poisson integrals imply that
\begin{equation*}
    d(\mu_{\alpha})_{a}(z) = \frac{1 - \lvert f(z) \rvert^{2}}{\lvert \alpha - f(z) \rvert^{2}}dm(z).
\end{equation*}

The Aleksandrov-Clark measures satisfy the following disintegration Theorem due to Aleksandrov, \cite{aleksandrov1987multiplicity}.

\begin{theorem}
    Let $h$ be a continuous function on the unit circle. Then
    \begin{equation*}
        \int_{\mathbb{T}} h(z) dm(z) = \int_{\mathbb{T}}\int_{\mathbb{T}} h(z) d\mu_{\alpha}(z)dm(\alpha).
    \end{equation*}
\end{theorem}

In fact, Aleksandrov showed that this can be extended to $L^{1}(\mathbb{T})$ functions in a suitable way. The Carathéodory angular derivatives of $f$ can be characterized using the Aleksandrov-Clark measures. We state the following realization of this, which is Lemma 3.4. in \cite{MR2394657}. It can also be deduced by combining (VI-7), (VI-9), and (VI-10) from \cite{MR1289670}.

\begin{lemma}\label{lemma:clark_ang_derivative}
    Let $f$ be a holomorphic self-map of the unit disc with Aleksandrov-Clark measure $\mu_{\alpha}$ and $I \subset \mathbb{T}$ a relatively open arc. Suppose $\mu_{\alpha}$ is absolutely continuous with respect to Lebesgue measure on $I$. Then $f$ has a Carathéodory angular derivative at $\lambda \in I$ if and only if
    \begin{equation*}
        \int_{I} \frac{d\mu_{\alpha}(z)}{\lvert z -\lambda \rvert^{2}} < \infty.
    \end{equation*}
\end{lemma}

\begin{remark}
    It was pointed out by the referee that it is enough to assume that $\mu_{\alpha}$ has no point masses on $I$ in the previous lemma. This is contained in Lemma 3.4. in \cite{MR2394657}.
\end{remark}

Let $H^{2}$ denote the Hardy space of the unit disc and $A^{\infty}$ be the class of holomorphic functions on the unit disc with smooth extension to the unit circle. As mentioned in the introduction Beurling-Carleson sets are the boundary zero sets $A^{\infty}$. Indeed, Beurling-Carleson sets are also the \textit{strong boundary zero sets} of such functions, in the sense that if $\varphi \in A^{\infty}$, then $\left\{ \lambda \in \mathbb{T} : (z-\lambda)^{-1}\varphi \in A^{\infty} \right\}$ is a Beurling-Carleson set. It is a Theorem of Makarov and Nikolski that countable unions of Beurling-Carleson sets are the strong boundary zero sets for $H^{2}$, \cite{MR0697433}.

\begin{theorem}\label{thm:Makarov_Nikolski}
    Let $\varphi \in H^{2}$. Then the set
    \begin{equation*}
                \left\{ \lambda \in \mathbb{T} : \frac{\varphi}{z-\lambda} \in H^{2} \right\}
    \end{equation*}
    is a countable union of Beurling-Carleson sets. Conversely, let $E \subset \mathbb{T}$ be a countable union of Beurling-Carleson sets. Then there exists an outer function $\varphi \in H^{2}$, such that
    \begin{equation*}
        E = \left\{ \lambda \in \mathbb{T} : \frac{\varphi}{z-\lambda} \in H^{2} \right\}.
    \end{equation*}
\end{theorem}

    We remark that if $\varphi \in H^{2}$, then the sets
    \begin{equation*}
        E_{n} = \left\{ \lambda \in \mathbb{T} : \lVert (z-\lambda)^{-1}\varphi \rVert_{2} \leq n \right\},
    \end{equation*}
    are Beurling-Carleson sets.

\section{Proof of the Main Result}\label{sec:proof_thm_1}

    We now turn to the proof of Theorem \ref{thm:local_ang_deriv}. We begin with a Lemma.

    \begin{lemma}\label{lemma:local_ac}
        Let $f$ and $I$ be as in Theorem \ref{thm:local_ang_deriv}. Then for Lebesgue almost-every $\alpha \in \mathbb{T}$ the restriction of $\mu_{\alpha}$ to $I$ is absolutely continuous.
        \begin{proof}
            Recall that the absolutely continuous part of $\mu_{\alpha}$ is given by
            \begin{equation*}
                d(\mu_{\alpha})_{a} = \frac{1-\lvert f \rvert^{2}}{\lvert 1 - \overline{\alpha}f \rvert^{2}}dm.
            \end{equation*}
            Thus for each $\alpha \in \mathbb{T}$
            \begin{equation*}
                \mu_{\alpha}(I) \geq (\mu_{\alpha})_{a}(I) = \int_{I} \frac{1-\lvert f(z) \rvert^{2}}{\lvert 1 - \overline{\alpha}f(z) \rvert^{2}}dm(z).
            \end{equation*}
            Integrating with respect to $\alpha$ and using Fubini's Theorem gives
            \begin{equation*}
                \int_{\mathbb{T}} \int_{I} d\mu_{\alpha}(z)dm(\alpha) \geq \int_{I} \int_{\mathbb{T}} \frac{1-\lvert f(z) \rvert^{2}}{\lvert 1 - \overline{\alpha}f(z) \rvert^{2}}dm(\alpha) dm(z).
            \end{equation*}
            Using the Aleksandrov disintegration Theorem on the left-hand side and that $\lvert f(z) \rvert < 1$ almost-everywhere on $I$ on the right-hand side gives
            \begin{equation*}
                \lvert I \rvert = \int_{\mathbb{T}} \int_{I} d\mu_{\alpha}(z)dm(\alpha) \geq \int_{I} \int_{\mathbb{T}} \frac{1-\lvert f(z) \rvert^{2}}{\lvert 1 - \overline{\alpha}f(z) \rvert^{2}}dm(\alpha) dm(z) = \lvert I \rvert.
            \end{equation*}
            Thus
            \begin{equation*}
                \int_{\mathbb{T}} \left( \mu_{\alpha}(I) - (\mu_{\alpha})_{a}(I) \right)dm(\alpha) = 0.
            \end{equation*}
            Since $\mu_{\alpha}(I) - (\mu_{\alpha})_{a}(I) \geq 0$ for all $\alpha$ it follows that $\mu_{\alpha}(I) = (\mu_{\alpha})_{a}(I)$ for almost-every $\alpha \in \mathbb{T}$ as claimed.
        \end{proof}
    \end{lemma}

    \begin{remark}
        From the proof of the Lemma we see that the conclusion remains valid if we merely assume that $\lvert f(z) \rvert < 1$ almost everywhere on $I$.
    \end{remark}
    We are now ready for the proof of Theorem \ref{thm:local_ang_deriv}.
    
    \begin{proof}[Proof of Theorem \ref{thm:local_ang_deriv}]
        Recall that $f$ is a holomorphic self-map on the unit disc. Suppose first that $f$ is locally non-extreme on the open interval $I$,
        \begin{equation*}
            \int_{I} \log (1-\lvert f(z)\rvert) dm(z) > -\infty.
        \end{equation*}
        We must show that $C(f) \cap I$ is a countable union of Beurling-Carleson sets. By Lemma \ref{lemma:local_ac} we may assume without loss of generality that $\mu = \mu_{1}$ is absolutely continuous on $I$. Indeed, if it is not we may replace $f$ by $\alpha f$ with a suitable $\alpha \in \mathbb{T}$. Under this assumption $f$ has a Carathéodory angular derivative at $\lambda \in \mathbb{T}$ if and only if
        \begin{equation*}
            \int_{\mathbb{T}} \frac{1}{\lvert z -\lambda \rvert^{2}}d\mu(z) < \infty,
        \end{equation*}
        by Lemma \ref{lemma:clark_ang_derivative}. For $\lambda \in I$ this is equivalent to local summability
        \begin{equation*}
            \int_{I}  \frac{1}{\lvert z -\lambda \rvert^{2}} \frac{1-\lvert f(z) \rvert^{2}}{\lvert 1 - f(z) \rvert^{2}} dm(z) = \int_{I} \frac{1}{\lvert z -\lambda \rvert^{2}}d\mu(z) < \infty.
        \end{equation*}
        Since $1-\lvert f \rvert$ is $\log$-integrable on $I$ there exists a unique outer function, $F \in H^{2}$, such that
        \begin{equation*}
            \lvert F(z) \rvert^{2} = \begin{cases}
                \frac{1-\lvert f(z) \rvert^{2}}{\lvert 1 - f(z) \rvert^{2}}, \; \; z \in I \\
                1, \; \; z \in \mathbb{T} \setminus I
            \end{cases}.
        \end{equation*}
        Then
        \begin{equation*}
            \int_{\mathbb{T}} \frac{\lvert F(z) \rvert^{2}}{\lvert z -\lambda \rvert^{2}}dm(z) = \int_{\mathbb{T}\setminus I} \frac{dm(z)}{\lvert z -\lambda \rvert^{2}} + \int_{I} \frac{d\mu(z)}{\lvert z -\lambda \rvert^{2}}.
        \end{equation*}
        Hence,
        \begin{equation*}
            \left\{ \lambda \in \mathbb{T} : F/(z-\lambda) \in H^{2} \right\} = C(f) \cap I.
        \end{equation*}
        Now by the Theorem \ref{thm:Makarov_Nikolski} the left-hand side is a countable union of Beurling-Carleson sets.\newline

        Suppose now that $E \subset \mathbb{T}$ is a countable union of Beurling-Carleson sets. We will construct a non-extreme holomorphic self-map of the unit disc, $f$, such that $C(f) = E$. By Theorem \ref{thm:Makarov_Nikolski} there exists an outer function $\varphi \in H^{2}$ of norm $1$, such that
        \begin{equation*}
            E = \left\{ \lambda \in \mathbb{T} : (z-\lambda)^{-1}\varphi \in H^{2} \right\}.
        \end{equation*}
        Let $f : \mathbb{D} \to \mathbb{D}$, $f(0) = 0$, be the unique holomorphic function, such that
        \begin{equation*}
            \frac{1+f(z)}{1-f(z)} = \int_{\mathbb{T}} \frac{\zeta + z}{\zeta -z} \lvert \varphi \rvert^{2}dm(\zeta).
        \end{equation*}
        It follows from Lemma \ref{lemma:clark_ang_derivative} that $f$ has Carathéodory angular derivative at $\lambda \in \mathbb{T}$ if and only if
        \begin{equation*}
            \int_{\mathbb{T}} \frac{\lvert \varphi(z) \rvert^{2}}{\lvert z -\lambda \rvert^{2}}dm(z) < \infty.
        \end{equation*}
        Thus, by construction, $C(f) = E$. It remains to check that that $f$ is non-extreme. Indeed, by standard properties of Poisson integrals,
        \begin{equation*}
            \frac{1-\lvert f(z) \rvert^{2}}{\lvert 1 - f(z) \rvert^{2}} = \lvert \varphi \rvert^{2}, \text{ for almost every } z \in \mathbb{T}.
        \end{equation*}
        Since $1-f(z)$ and $\varphi$ both belong to $H^{2}$ and hence are $\log$ integrable it follows that
        \begin{equation*}
            \int_{\mathbb{T}} \log \left( 1-\lvert f(z) \rvert \right) dm(z) > -\infty,
        \end{equation*}
        as claimed.
    \end{proof}
    
\bibliographystyle{abbrv}
\bibliography{main}

\end{document}